\documentclass[12pt,a4paper]{amsproc}
\usepackage{amsfonts}
\usepackage{mathrsfs, amssymb, amsmath, graphicx}

\newtheorem{theorem}{Theorem}

\def\<{\langle}\def\>{\rangle}

\def\B{\mathcal{B}}
\def\A{\mathscr{A}}
\def\D{\mathscr{D}}

\def\H{\mathcal{H}}

\def\R{\mathcal{R}}

\def\leq{\leqslant}
\def\geq{\geqslant}
\def\phi{\varphi}

\begin{document}
\title[A note on commutators]{A note on commutators in algebras \\ of unbounded operators}

\author{Richard V. Kadison}
\address{D.K., Mathematics Department, University of Pennsylvania, U.S. }
\email{kadison@math.upenn.edu}

\author{Zhe Liu}
\address{Z.L., DKU Faculty, Duke Kunshan University, China }
\email{zhe.liu372@dukekunshan.edu.cn}

\author{Andreas Thom}
\address{A.T., Institute of Geometry, TU Dresden, Germany }
\email{andreas.thom@tu-dresden.de}

\maketitle

\begin{abstract}
\noindent We show that the identity is the sum of two commutators in the algebra
of all operators affiliated with a von Neumann algebra of type II$_1$,
settling a question, in the negative, that had puzzled a number of us.
\end{abstract}

\section{Introduction}
A commutator in a ring, an algebraic system, much studied throughout mathematics, that is endowed with operations of addition and multiplication having the usual properties found in the ring of integers (but not necessarily commutativity of multiplication) is an element of the form $ab-ba$, where $a$ and $b$ are elements of the ring. This commutator is often referred to as the (Lie) bracket of $a$ and $b$, and denoted by $[a,b]$ (occasionally as $[b,a]$). Our primary focus is the family of `rings of operators' (acting on a Hilbert space), introduced by von Neumann in [v.N.]. (They were re-named `von Neumann algebras' by J. Dixmier in [Dix] at the suggestion of J. Dieudonn$\rm \acute{e}$.) Commutators and bracketing appear in the study of Lie algebras (and, hence, Lie groups) [C] and in the Newton-Hamilton formulation of classical analytical dynamics.

In [K;S] it is proved that each derivation of a von Neumann algebra is
inner -- the assertion that $H^1(\R,\R)=(0)$, the first cohomology
group of a von Neumann algebra $\R$ with coefficients in itself is $(0)$.
Specifically, each derivation $\delta$ of $\R$ into $\R$ is of the form
$\mathrm{Ad}(B)$, for some $B$ in $\R$; that is, $\delta(A)=\mathrm{Ad}(B)
(A)=AB-BA$. So, the range of $\delta$ consists of commutators ($AB-BA=\delta(A)$).
This suggests extending many of the commutator theorems to theorems about
derivations and their ranges, especially in situations where derivations
need not be inner ($C^*$-algebras, for example). An instance of this can
be found in Sakai's exciting book [S1] and the classic Jacobson paper [J].
In [J] (Lemma 2), it is proved that if $[A,B]$ commutes with $A$, then
$[A,B]$ is a nilpotent. In [S1] (Theorem 2.2.7), it is proved that if
$\delta$ is a derivation of a $C^*$-algebra and $\delta(A)$ commutes
with $A$, where $A$ is a normal element, then $\delta(A)=0$.
The fact that $A$ is normal leads us from ``nilpotent" to 0. Of course,
commutators and the Jacobson result form the background for Sakai's
clever arguments. This discussion makes it evident that derivations
and commutators are closely allied subjects. This is even more evident
when we consider the foundational aspects of quantum dynamics.

The Heisenberg commutation relation, the fundamental relation of
quantum mechanics, $QP-PQ=i\hbar I$, where $Q$ is a position
observable for a particle of the quantum mechanical system,
$P$ is its (conjugate) momentum, and $\hbar=h/2\pi$, where $h$ is
Planck's experimentally determined quantum of action (approximately
$6.625 \times 10^{-27}$ erg sec), is clearly an assertion about a
commutator, $QP-PQ$. It informs us that any mathematical model
suitable for a presentation of quantum physics must be {\it noncommutative}.
At the same time, by virtue of important mathematical arguments (see [Wie]
and [Win]), we are assured that finite matrices are not suitable for the
entire model and even {\it bounded} operators on a Hilbert space will not do.
The model of choice for our quantum mechanical systems are the algebras
of operators on a Hilbert space; especially the $C^*$-algebras and the
von Neumann algebras and their ``affiliated" ({\it unbounded}\/) operators. This
will be discussed further in the next section. Dirac's program [D] associates the bounded observables of some quantum
mechanical system with the self-adjoint operators in a von Neumann algebra
$\R$.  The symmetries of the system (and the associated
conservation laws) correspond to the symmetry
groups as groups of automorphisms of $\R$.  The time-evolution of the
system, with a given dynamics, corresponds to a one-parameter group of
automorphisms, $t\to\alpha_t$ of $\R$.  Loosely speaking, $\alpha_t$
will be $\exp(it\delta)$ (thought of in series terms) for some linear mapping $\delta$ (of the ``algebra"
of observables). Thus $\frac{d(\alpha_t(A))}{dt}\big|_{t=0}=-iHA+iAH=i[A,H]$.
At the same time, $\frac{d(\alpha_t(A))}{dt}\big|_{t=0}=i\delta(A)$. Hence $\delta(A)=[A,H]$.
Here, $H$ is the (quantum) Hamiltonian of the system, which corresponds to
the total energy of the system as the classical Hamiltonian does for a classical
mechanical system. This quantum Hamiltonian ``generates" the time-evolution of
the quantum system, again, as it does for the classical system. The one-parameter
group of automorphisms is ``implemented" by the one-parameter unitary group
$t\to e^{itH}$. (A ``world" of Mathematics underlies all this, including, prominently,
Stone's Theorem [St] and Wigner's Theorem (see [B])  -- and again, cohomology, this time,
of groups.)
For Hamiltonian mechanics, time-differentiation of the dynamical variable is Poisson bracketing with the
Hamiltonian.  In quantum mechanics, differentiation of the
``evolving observable" is Lie bracketing with the (quantum) Hamiltonian. We see that bracketing is a derivation of the system as are the other
generators of the one-parameter automorphism groups that describe our physical system and its symmetries. Of course, we must study those derivations.
\section{Some historical comments}

We have just noted how closely allied the subjects of commutators and derivations
are. At the same time, we indicated some of the ways in which those two subjects
were intertwined. In this section, we give a brief description of the way in which
``Commutators" has developed as an independent mathematical subject. It is interesting
to observe how pure mathematics and pure mathematicians will often take subjects
that originate and develop together and separate them through the process of
abstraction. It's a tribute to pure mathematics and abstraction that the independent
developments of each subject often lead to astonishing results in each, far beyond the
reach of the joint development. The complexity and vitality of this independent
development is often testified to by the inability of practitioners of each of the
subjects to converse with one another about their subjects. One sees this very
clearly, not only where the subjects of commutators and derivations are concerned,
but also when the vast area of mathematical physics, which uses both
commutators and derivations intimately is taken into account; witness our
discussion of the Heisenberg relation, bracketing with the Hamiltonian, and
time-evolution of physical systems.

Perhaps, the first, and the most prominent, result of the theory of commutators is Shoda's
1936 theorem [Sh]: in the algebra of all matrices over a field of characteristic 0,
a matrix has trace 0 if and only if it is a commutator. This theorem was
generalized to matrices over arbitrary fields by A. Albert and
B. Muckenhoupt in 1957 [AM]. Halmos [H] shows that every operator in $\B(\H)$,
where $\H$ is an infinite-dimensional Hilbert space, is a sum of two commutators.
In the same article, Halmos raises questions about the spectra of commutators, which
leads directly to the very important result of Brown and Pearcy [BP1] telling us that
the commutators in $\B(\H)$ are all the operators {\it not} of the form $\lambda I+K$,
where $K$ is compact and $\lambda\neq 0$. (See, also, [FK] for commutators in a
II$_1$ factor, where it is proved that the trace lies in the closed convex hull
of the spectrum, which certainly limits the geometric possibilities for that spectrum.)

In connection with other algebras of operators on a Hilbert space ($C^*$-algebras
and von Neumann algebras), where a ``trace" may be present, as in the
case of finite factors, clearly the trace of a commutator or sums of commutators
will be 0. Is the converse true in a finite factor? Is each trace 0 element a
commutator, or a sum of commutators? As starkly as
they present themselves, these are not easy questions. They have been asked
since 1950. Shoda's result answers this question affirmatively for factors
of type I$_n$ (with $n$ finite, {\it viz.} finite matrix algebras over $\mathbb{C}$,
the complex numbers). It is one of the major questions where factors of type II$_1$
are concerned. For sums of commutators, some splendid work has been done
in recent years by a number of people. Beginning with the seminal work
of Fack and de la Harpe, who show that each trace 0 element is a sum of ten
commutators or fewer, the number of commutators needed for a trace 0 operator
has been steadily decreased. In recent years, Marcoux [M] has come, seemingly,
close to ending the search by showing that two commutators will suffice. But
recall, that even the identity is a sum of two commutators in $\B(\H)$
with $\H$ infinite-dimensional, as Halmos shows us [H], without its being a single
commutator as Wintner and Wielandt show us ([Win], [Wie]). A number of
interesting and important commutator results for special operators
and special classes of operators add valuable information to our
knowledge about these commutator questions. The very nice work of
Dykema and Skripka [DS] is a case in point. Among other things, they
show that trace 0 normal operators, in a II$_1$ factor, with point spectrum are commutators
as are nilpotent operators.

In very recent times, these commutator questions have been raised in
the context of algebras that contain unbounded as well as bounded
operators. In [Z], the family of operators affiliated with a factor
of type II$_1$ is shown to be an associative (non-commutative) * algebra,
so that commutators, derivations and such have very clear meanings.
We call these and those affiliated with finite von Neumann algebras,
in general, Murray-von Neumann algebras. In [Z], it is proved that
the identity cannot be a commutator of two self-adjoint elements in
the Murray-von Neumann algebra affiliated with a II$_1$ factor. This
result is extended to general Murray-von Neumann algebras, those affiliated
with a finite von Neumann algebra, and at the same time, it is noted
that nonzero scalar multiples of the identity cannot be a commutator
of two operators either one of which is self-adjoint. It is still not
known whether the identity is a commutator of two arbitrary elements in
these Murray-von Neumann algebras. In [KL], the following was conjectured:

{\it Let $\R$ be a finite von Neumann algebra. If $p$ is a non-commutative polynomial in $n$ variables
 with the property that, whenever the variables are replaced by operators
 in $\R$ the resulting operator in $\R$ has trace 0, then, whenever replacing
 the variables in $p$ by operators in the Murray-von Neumann algebra affiliated with
 $\R$ produces a bounded operator,
 necessarily in $\R$, that operator has trace 0.}

In the next section, we prove that this conjecture has a negative answer by
showing that the identity is the sum of two commutators in each Murray-von Neumann algebra
affiliated with von Neumann algebra of type II$_1$. An article by Dykema and
Kalton has information highly relevant to the result just mentioned [DK; see
especially, Lemma 4.3 and Corollary 4.8].

The classical foundational representation (irreducible) of the Heisenberg relation
is effected by unbounded operators affiliated with $\B(\H)$ (see [Z] and [KL]).
Results corresponding to Brown and Pearcy's for $\B(\H)$, were found by Halpern [Ha]
for factors of
type II$_\infty$. Brown and Pearcy [BP2], themselves, show
that every operator in a type III factor is a commutator.

\section{Sums of commutators}
The notation `$V_{\R}(E,F)$' denotes a partial isometry in the von Neumann algebra $\R$ with initial projection $E$ and final projection $F$ (both, necessarily in $\R$). There are a number of algebraic, operator identities associated with this notation:
\begin{align}
&V_{\R}(E,F)^*V_{\R}(E,F)=E,
\nonumber\\&V_{\R}(E,F)V_{\R}(E,F)^*=F,
\nonumber\\&V_{\R}(E,F)=V_{\R}(E,F)E=FV_{\R}(E,F),
\nonumber\\&V_{\R}(E,F)^*=EV_{\R}(E,F)^*=V_{\R}(E,F)^*F.\nonumber
\end{align}
If the von Neumann algebra $\R$ under consideration is clear from the context, we write `$V(E,F)$' omitting the subscript `$\R$.'

Suppose $\{E_1,...,E_n\}$ and $\{F_1,...,F_n\}$ are, each, a mutually orthogonal family of $n$ projections in $\R$, where $n$ is possibly, $\aleph_0$; and suppose that $E_j\sim F_j$ (mod $\R$). Then there is a partial isometry $V(E_j,F_j)$ in $\R$, for each $j$ in $\{1,...,n\}$, and $\sum_{j=1}^nV(E_j,F_j)$ is a partial isometry $V(E,F)$ in $\R$, where $E=\sum_{j=1}^nE_j$ and $F=\sum_{j=1}^nF_j$. To see this, note that when $n$ is finite,
\begin{align}
\left(\sum_{j=1}^nV(E_j,F_j)\right)^*\left(\sum_{k=1}^nV(E_k,F_k)\right)&=\sum_{j,k=1}^nV(E_j,F_j)^*V(E_k,F_k)
\nonumber\\&=\sum_{j,k=1}^nV(E_j,F_j)^*F_jF_kV(E_k,F_k)
\nonumber\\&=\sum_{j=1}^nV(E_j,F_j)^*F_jV(E_j,F_j)
\nonumber\\&=\sum_{j=1}^nV(E_j,F_j)^*V(E_j,F_j)
\nonumber\\&=\sum_{j=1}^nE_j=E.\nonumber
\end{align}
Similarly, $\big(\sum_{k=1}^nV(E_k,F_k)\big)\big(\sum_{j=1}^nV(E_j,F_j)\big)^*=F$.

Of course, $V(E,F)$ is one of a set of partial isometries in $\R$, not some definite element in that set. So, $V(E,F)^*$ is a possible choice for $V(F,E)$, usually the choice we shall make when $V(E,F)$ has been chosen.

When $E_1+\cdots+E_n=I=F_1+\cdots+F_n$, $\sum_{j=1}^nV(E_j,F_j)$ is a partial isometry in $\R$ with initial and final projection, both, $I$. Thus $\sum_{j=1}^nV(E_j,F_j)$, which we may reasonably choose for $V(I,I)$, in this instance, is a unitary operator $U$ in $\R$. We note that, in this case, from the identities for partial isometries,
\begin{align}
UE_rU^*&=\left(\sum_{j=1}^nV(E_j,F_j)\right)E_r\left(\sum_{k=1}^nV(E_k,F_k)\right)^*
\nonumber\\&=\left(\sum_{j=1}^nV(E_j,F_j)\right)E_rV(E_r,F_r)^*
\nonumber\\&=\left(\sum_{j=1}^nV(E_j,F_j)\right)V(E_r,F_r)^*
\nonumber\\&=V(E_r,F_r)V(E_r,F_r)^*=F_r, \quad (r\in\{1,...,n\}).\nonumber
\end{align}

The preceding discussion applies when our sums are infinite (that is, when $n=\infty$). In this case, however, we must apply the discussion (as is) to the finite partial sums, and then, apply the results of that application to our infinite sums by taking vector limits of finite sums of vectors and strong-operator limits of finite sums of operators.

We shall make use of these considerations for the purpose of constructing sums of commutators in certain von Neumann algebras, in particular, type II$_1$ von Neumann algebras $\R$ and their
algebras of affiliated operators, denoted by `$\A_{\rm f}(\R)$,' which we call Murray-von Neumann algebras.
\begin{theorem}
Let $\R$ be a von Neumann algebra of type II$_1$. Then the identity operator $I$ in $\R$ is the sum of two commutators in $\A_{\rm f}(\R)$.
\end{theorem}
\begin{proof}
Toward demonstrating that the identity operator $I$ in $\R$ is the sum of two commutators in $\A_{\rm f}(\R)$, the Murray-von Neumann algebra of operators affiliated with $\R$, we construct two sequences of mutually orthogonal projections $\{E_n\}_{n=1}^\infty$ and $\{F_n\}_{n=1}^\infty$ in $\R$ such that

(i) $\vee_{n=1}^\infty E_n=I$, \quad $E_1\sim I-E_1$,

 (ii) $\vee_{n=1}^\infty F_n=E_1$,

 (iii) $\tau(E_n)=\frac{1}{2^n}I$,\quad $\tau(F_n)=\frac{1}{2^{n+1}}I$ \quad($n=1,2,...$),

\noindent where $\tau$ is the normalized, center-valued trace ($\tau(I)=I$) on $\R$. This construction is effected by means of repeated ``bisection"
of projections in $\R$. We, first, bisect $I$ to yield $E_1$, equivalent to $I-E_1$ (mod $\R$), then bisect $I-E_1$ to yield $E_2$, equivalent to $I-E_1-E_2$, and so forth. Now, we bisect $E_1$ to yield $F_1$, equivalent to $E_1-F_1$, then bisect $E_1-F_1$ to yield $F_2$, equivalent to $E_1-F_1-F_2$, and so forth. (see [K-R; Lemma 6.5.6].) We, next, bisect each $E_n$ to yield projections $E_n'$ and $E_n''$ (in $\R$) with $E_n'E_n''=0$, $E_n'+E_n''=E_n$ and $E_n'\sim E_n''$ (mod $\R$). For $n=1$, note that
$$
\tau(E_1')=\tau(E_1'')=\frac{1}{4}I
$$
and
$$
\tau(E_2)=\frac{1}{2^2}I=\frac{1}{4}I, \quad\tau(F_1)=\frac{1}{2^{1+1}}I=\frac{1}{4}I.
$$
Therefore, there are partial isometries $V(E_1',E_2)$ and $V(E_1'',F_1)$. Define a partial isometry $V_1$ by
$$
V_1=V(E_1',E_2)+V(E_1'',F_1).
$$
(Note that, in the notation just described, we can choose $V_1$ for $V(E_1, E_2\vee F_1)$.) For $n\geq 2$, note that
$$
\tau(E_n')=\tau(E_n'')=\frac{1}{2}\tau(E_n)=\frac{1}{2^{n+1}}I
$$
and for all $n$, we have
$$
\tau(E_n)=\frac{1}{2^n}I, \quad \tau(F_n)=\frac{1}{2^{n+1}}I.
$$
There are partial isometries $V(E_n',E_{n+1})$ and $V(E_n'',F_n)$. Define a partial isometry $V_n$ by
$$
V_n=V(E_n',E_{n+1})+V(E_n'',F_n),\quad (n\geq 2).
$$
(Note, again, that we can choose $V_n$ for $V(E_n, E_{n+1}\vee F_n)$.)

Now, let $U_1$ be $\sum_{n=1}^\infty V_n$. We note that $\sum_{n=1}^\infty E_n=I$ and that
$$\sum_{n=1}^\infty E_{n+1}\vee F_n=\sum_{n=1}^\infty F_n+\sum_{n=1}^\infty E_{n+1}=E_1+\sum_{n=1}^\infty E_{n+1}=\sum_{n=1}^\infty E_{n}.
$$
For this, we have that $\{E_n\}$ are mutually orthogonal projections, whence  $\sum_{n=1}^\infty E_{n}\leq I$. But
$$
\tau\left(\sum_{n=1}^k E_{n}\right)=\sum_{n=1}^k\tau(E_n)=\sum_{n=1}^k \frac{1}{2^n}I\leq \tau\left(\sum_{n=1}^\infty E_{n}\right)\leq\tau(I)=I,
$$
for all positive integer $k$. Thus $\tau(\sum_{n=1}^\infty E_{n})=I$ and $\sum_{n=1}^\infty E_{n}=I$. Therefore, $U_1(=\sum_{n=1}^\infty V_n)$ is a unitary operator. Similarly, let $U_2$ be $\sum_{n=1}^\infty W_n$, where
$$
W_n=V(E_n',F_n)+V(E_n'',E_{n+1}),\quad (n=1,2,...).
$$
Then $U_2$ is a unitary operator.

We shall define an operator $A$, affiliated with $\R$, on $\H$. We start with $A_0$, where loosely,
$$
A_0=1E_1+2E_2+\cdots+2^{n-1}E_n+\cdots
$$
More precisely, for each finite linear combination $a_1x_1+a_2x_2+\cdots+a_nx_n(=x)$ with $x_j$ in $E_j(\H)$,
$$
A_0x=1a_1x_1+2a_2x_2+\cdots+2^{n-1}a_nx_n.
$$
Since $\sum_{n=1}^\infty E_n=I$, $A_0$ is densely defined. From [K-R; Lemma 5.6.1], $A_0$ has a self-adjoint closure $A$. Note that $\cup_{n=1}^\infty E_n(\H)(=\D_0)$ is a core for $A$ such that $U^*AUx=Ax$ for each $x$ in $\D_0$ and each unitary $U$ in $\R'$. From [K-R;Remark 5.6.3], $A$ is affiliated with $\R$.

Now, for each $x$ in $E_n'(\H)$ ($n=1,2,...$), we have that
\begin{align}
2Ax-{U_1}^*AU_1x-{U_2}^*AU_2x&=2\cdot 2^{n-1}x-{U_1}^*(2^n(U_1(x)))-{U_2}^*(1(U_2(x)))
\nonumber\\&=2^nx-2^nx-1x=-Ix.
\nonumber
\end{align}
It follows that on $E_n'(\H)$,
$$
I={U_1}^*AU_1-A+{U_2}^*AU_2-A=[AU_1,-{U_1}^*]+[AU_2,-{U_2}^*].
$$
The symmetry applies when $x$ is in $E_n''(\H)$. Since $\sum_{n=1}^\infty E_n=I$, the above identity holds on $\cup_{n=1}^\infty E_n(\H)(=\D_0)$, a core for $A$. As $I$ is bounded, the operator $[A\ \hat{\cdot}\ U_1,-{U_1}^*]\ \hat{+}\ [A\ \hat{\cdot}\ U_2,-{U_2}^*]$ has a unique bounded extension, which is, of course, $I$.
\end{proof}

\begin{theorem}
Let $\R$ be a von Neumann algebra with no non-zero finite central projection. Then the identity operator $I$ in $\R$ is the sum of two commutators in $\R$.
\end{theorem}
\begin{proof}
Our hypothesis assures us that $\R$ is ``properly infinite" [K-R II; Definition 6.3.1]. Thus $I$ is properly infinite in $\R$. The ``halving" lemma applies [K-R II; Lemma 6.3.3] and there is a projection $E$ in $\R$ such that $E$ is equivalent to $I$ and $I$ is equivalent to $I-E$. Let $V$ and $W$ be (partial) isometries in $\R$ such that $V^*V=I=W^*W$ and $VV^*=E$, $WW^*=I-E$. Then
$$
V^*V-VV^*+W^*W-WW^*=I-E+I-(I-E)=I.
$$\end{proof}

We recall, again, that the question of whether or not the identity of the Murray-von
Neumann algebra of a II$_1$ factor is a commutator, is still open.

\section*{Acknowledgments}

The third-named author acknowledges funding by the ERC Consolidator Grant No.~681207.

\end{document}